\documentclass[10pt]{amsart}

\usepackage{amsmath}    
\usepackage{graphicx}   
\usepackage{verbatim}   
\usepackage{color}      
\usepackage{subfigure}  
\usepackage{hyperref}   
\usepackage{fullpage}
\usepackage{bbm}
\usepackage{dsfont}
\usepackage{amssymb}

\usepackage{amssymb,amsmath}
\usepackage{amsthm}

\newtheorem*{thmnn}{Theorem}

\newtheorem*{mt}{Main Theorem}
\newtheorem{prop}{Proposition}
\newtheorem{lemma}{Lemma}

\newtheorem{cor}{Corollary}

\newcommand{\Q}{\mathbb{Q}}
\newcommand{\Z}{\mathbb{Z}}
\newcommand{\R}{\mathbb{R}}
\newcommand{\C}{\mathbb{C}}

\newcommand{\sym}{\mathrm{sym}}
\newcommand{\res}{\mathop{\rm res}}
\newcommand{\real}{\mathop{\rm Re}}
\newcommand{\imag}{\mathop{\rm Im}}

\title{Transition Mean Values of Shifted Convolution Sums}
\author{Ian Petrow}
\begin{document}
\maketitle

\begin{abstract}
Let $f$ be a classical holomorphic cusp form for $SL_2(\Z)$ of weight $k$ which is a normalized eigenfunction for the Hecke algebra, and let $\lambda(n)$ be its eigenvalues.  In this paper we study ``shifted convolution sums'' $\sum_n \lambda(n) \lambda(n+h)$ after averaging over many shifts $h$ and obtain asymptotic estimates.  The result is somewhat surprising: one encounters a transition region depending on the ratio of the square of the length of the average over $h$ to the length of the shifted convolution sum.  The phenomenon is similar to that encountered by Conrey, Farmer and Soundararajan in their 2000 paper \cite{nondifferentiable} on transition mean values of the Jacobi symbol, and the connection of both results to Eisenstein series and multiple Dirichlet series is discussed.
\end{abstract}

\vspace{10pt}

Let $f$ be a classical holomorphic Hecke eigencuspform for $SL_2(\Z)$ of weight $k\in 2 \Z, k>0$.  It admits a Fourier expansion of the form \[f(z) = \sum_{n\geq 1} n^{\frac{k-1}{2}}\lambda_f(n) e(nz),\] where the $\lambda_f(n)$ are the Hecke eigenvalues of $f$, normalized so that $\lambda_f(1) = 1$, and we use the standard notation $e(z) :=e^{2 \pi i z}$.  
In this paper, we study ``shifted convolution sums'', i.e. sums of the form \[ \sum_{n} \lambda_f(n)\lambda_f(n+h).\]  These sums have many applications in analytic number theory.  They often arise in subconvexity results for $GL_1$ and $GL_2$ $L$-functions, see Sarnak \cite{SarnakSCS}, or Lecture 4 of Michel in \cite{MichelLfcnsinfamilies}, and moreover, they are crucial to Soundararajan and Holowinsky's resolution of the quantum unique ergodicity conjecture for classical holomorphic modular forms, see \cite{QUE}.  In this paper we obtain asymptotic estimates for shifted convolution sums after averaging over many shifts $h$, i.e. we study sums of the form \begin{equation}\label{wavy}\sum_{h\asymp Y} \sum_{n \asymp X} \lambda_f(n) \lambda_f(n+h),\end{equation} where the notation $n \asymp X$ indicates a sum over $n$ of length $X$ with a for-now-unspecified smoothing.

We show that when $X$ and $Y$ grow large in such a manner that $Y^2/X \rightarrow \infty$, the double sum \eqref{wavy} has an asymptotic formula with well controlled error terms, and we obtain nontrivial asymptotic upper bounds when $Y^2/X \rightarrow 0$.  However, the most interesting case of our result is the transition phase between these two situations, that is, when $X$ and $Y$ go to infinity and $Y^2/X = c$, a fixed constant.  In this situation, the asymptotic we prove depends delicately on the constant $c$.  One perspective is to interpret the sum \eqref{wavy} as varying on the open first quadrant of the $(X,Y^2)$-plane.  Asymptotic estimates as $X$ and $Y^2$ go to infinity are a description of the singularity at infinity in this quarter-plane.  In this paper we find that the asymptotic behavior varies continuously on a blowup of the point at infinity in the quarter-$(X,Y^2)$-plane.  Our results for shifted convolution sums are very similar to the interesting results of Conrey, Farmer and Soundararajan on transition mean values of the Jacobi symbol \cite{nondifferentiable}.  They study the sum \[S(X,Y) := \sum_{\substack{m \leq X \\ m \text{ odd}}} \sum_{\substack{n \leq Y \\ n \text{ odd}}}\left(\frac{m}{n}\right),\] and similarly find asymptotic formulae when one of either $X$ or $Y$ grow much faster than the other, and a transition region when $X/Y $ is a fixed constant, in which the asymptotic varies continuously on a blowup.  The sums \eqref{wavy} and $S(X,Y)$ at first look dissimilar.  Note however that the problem of determining the asymptotic behavior of \eqref{wavy} is equivalent to determining asymptotic behavior in the sum \[\sum_{n \asymp X}\left(\sum_{m \asymp Y} \lambda_f(n+m)\right)^2,\] which puts \eqref{wavy} and $S(X,Y)$ on equal footing, and makes apparent why the transition region is at $Y^2/X = c$.  Moreover, we will see later that these two examples rest on the same ideas and have many common features.  First, however, we state precisely a corollary of the Main Theorem (found in section 4) of this paper.  The Main Theorem is very similar to the corollary, but allows a choice of cut-off functions.

We define the function $c_f$ on the positive real numbers \[ c_f(\alpha) := \frac{\pi^\frac{3}{2}}{2}  \alpha \sum_{n \geq 1} \lambda_f(n)^2 W_k\left(\pi^2 n \alpha \right), \] with \[W_k(x) :=  \frac{1}{2 \pi i}\int_{(1+a)} \frac{\Gamma(s+k-1) \Gamma\left(s-\frac{1}{2}\right)}{\Gamma(2-s)}x^{-s}  \,ds\] for any fixed $a>0$.  
\begin{cor}
If $1 \leq Y \leq X$ then
\[\sum_{h \leq Y} \sum_{n \geq 1} \lambda_f(n) \lambda_f(n+h) \left(\frac{n(n+h)}{X^2}\right)^{\frac{k-1}{2}} e^{-\frac{n+h}{X}} =\left(c_f\left( \frac{Y^2}{X} \right)-  \frac{\Gamma(k)L(1,\sym^2 f)}{2 \zeta(2)}\right) X +O_k\left(X^{\frac{1}{2}}Y^{\frac{1}{3}(1+\theta)}\right)\] where $c_f(\alpha)$ defined above is a smooth function on the positive real numbers which \begin{itemize} \item as $\alpha \rightarrow \infty$ decays faster than any polynomial \item as  $\alpha \rightarrow 0$ is
\[ = \frac{\Gamma(k)L(1,\sym^2 f)}{2\zeta(2)} + E_k(\alpha),\] where $E_k(\alpha) = o_k\left(\alpha^{\frac{1}{2}}\right)$ unconditionally, and $E_k(\alpha) = O_{k,\varepsilon}\left(\alpha^{\frac{3}{4}-\varepsilon}\right)$ for any $\varepsilon>0$ assuming the Riemann hypothesis for the classical Riemann zeta function. \end{itemize} and where $\theta = 0$ or $= \frac{7}{64}$ depending on whether one assumes the generalized Ramanujan conjecture for Maass forms of $SL_2(\Z)$ or not.
\end{cor}

As a benchmark, the best point-wise estimates for shifted convolution sums on $SL_2(\Z)$ give \begin{equation}\label{SCP} \sum_{n \asymp X} \lambda_f(n) \lambda_f(n+h) \ll_\varepsilon X^{\frac{1}{2}+\varepsilon},\end{equation} see Sarnak \cite{SarnakSCS}.  In Corollary 1, observe that if $Y^2$ is large compared to $X$ we have that \[\sum_{h \leq Y} \sum_{n \geq 1} \lambda_f(n) \lambda_f(n+h) \left(\frac{n(n+h)}{X^2}\right)^{\frac{k-1}{2}} e^{-\frac{n+h}{X}} \sim -  \frac{\Gamma(k)L(1,\sym^2 f)}{2 \zeta(2)} X \] and if $Y^2$ is small compared to $X$ then  \[ \sum_{h \leq Y} \sum_{n \geq 1} \lambda_f(n) \lambda_f(n+h) \left(\frac{n(n+h)}{X^2}\right)^{\frac{k-1}{2}} e^{-\frac{n+h}{X}} = o_k\left( X^{\frac{1}{2}}Y\right) + O_k\left(X^\frac{1}{2}Y^{\frac{1}{3}(1+\theta)}\right),\] or an even better bound if we assume the Riemann Hypothesis.  In the transition region when $Y^2$ is a constant multiple of $X$, the asymptotic growth is controlled by the function $c_f(\alpha)$.  

Now we describe the work of Conrey, Farmer and Soundararajan.  Their result is \begin{thmnn}[Conrey, Farmer and Soundararajan] Uniformly for all large $X$ and $Y$, we have \[S(X,Y) = \frac{2}{\pi^2} C\left( \frac{Y}{X}\right) X^{\frac{3}{2}} + O\left((XY^{\frac{7}{16}}+YX^{\frac{7}{16}})\log XY\right),\] where for $\alpha \geq 0$ we define \[C(\alpha) = \sqrt{\alpha} + \frac{1}{2 \pi} \sum_{k=1}^\infty \frac{1}{k^2} \int_0^\alpha \sqrt{y} \left(1-\cos \left(\frac{2 \pi k^2}{y}\right) + \sin \left(\frac{2 \pi k^2}{y}\right)\right)\,dy.\] An alternate expression for $C(\alpha)$ is \[C(\alpha) = \alpha+ \alpha^\frac{3}{2} \frac{2}{\pi} \sum_{k=1}^\infty \frac{1}{k^2} \int_0^{\frac{1}{\alpha}} \sqrt{y} \sin \left(\frac{\pi k^2}{2y}\right)\,dy.\]\end{thmnn}   From the first expression, one finds that after integrating by parts that \[C(\alpha) = \sqrt{\alpha} + \frac{\pi}{18}\alpha^{\frac{3}{2}}+O\left(\alpha^{\frac{5}{2}}\right)\] as $\alpha \rightarrow 0$.  The second expression gives the limiting behavior \[C(\alpha) = \alpha+ O\left(\alpha^{-1}\right)\] as $\alpha \rightarrow \infty$.  

We make some brief remarks on notation: in this paper we use an integral with a subscript in parentheses $\int_{(c)}$ to denote the contour integral in the complex plane with positive orientation along the line $\real(s) = c$, and use the standard notations $\ll_x $, $o_x(\cdot)$ and $O_x(\cdot)$, where the subscript denotes that the implied constants depend on the parameter $x$ and are otherwise absolute.  We have made an effort to be explicit about the dependence of our error terms on the various parameters involved in this paper.  

Finally, we would like to thank the number theory community at Stanford for the stimulating academic environment, in particular, Robert Rhoades, Xiannan Li, Bob Hough, David Sher, Akshay Venkatesh and above all, Kannan Soundararajan, discussions with whom influenced the final form of this paper.  

\section{Connections between the results via Eisenstein and Multiple Dirichlet Series}

There are many parallels between our work and that of Conrey, Farmer and Soundararajan, and moreover both results can be interpreted as averages of Fourier coefficients of Eisenstein series.  We first mention some interesting similarities.

In both our work and that of Conrey, Farmer and Soundararajan, the result when one parameter grows more rapidly than the other can be deduced by evaluating the long sum first and then making trivial estimates on the short sum.  Indeed, if $Y^2$ is very small compared to $X$, Sarnak's solution to the shifted convolution problem \eqref{SCP} yields \[\sum_{h \leq Y}\sum_{n \leq X} \lambda_f(n) \lambda_f(n+h) \ll_{k,\varepsilon} X^{\frac{1}{2}+\varepsilon}Y.\] In the case of the Jacobi symbol, one can use the Polya-Vinogradov inequality to show that \[\sum_{\substack{n \leq X \\ n \text{ odd}}} \left(\frac{n}{m}\right) = \begin{cases} \frac{X}{2} \frac{\varphi(m)}{m} + O_\varepsilon\left(X^\varepsilon\right) & \text{if } m = \square \\ O\left(m^\frac{1}{2}\log m \right) & \text{if } m \neq \square \end{cases} \] from which it follows that \begin{eqnarray*} S(X,Y) & = & \sum_{\substack{m \leq Y\\ m = \text{odd } \square}} \left( \frac{X}{2} \frac{\varphi(m)}{m}  + O_\varepsilon\left(X^\varepsilon\right)  \right) + O\left(Y^{\frac{3}{2}}\log Y \right) \\ & = & \frac{2}{\pi^2} XY^\frac{1}{2} + O_\varepsilon \left(Y^{\frac{3}{2}}\log Y + Y^\frac{1}{2}X^\varepsilon + X \log Y\right),\end{eqnarray*} and similarly if the roles of $n,m$ are reversed, see \cite{nondifferentiable}.  The case of shifted convolution sums when $Y^2$ is much larger than $X$ follows even more simply.  Indeed, we have that \[\sum_{n \leq X} \lambda_f(n) \ll_f X^{\frac{1}{3}}.\] See for example, \cite{HafnerIvic}.  Then we find 
\[\sum_{h \leq X} \sum_{n\leq X-h} \lambda_f(n) \lambda_f(n+h) = -\frac{L(1,\sym^2 f)}{2\zeta(2)}X +  O_f\left(X^{\frac{2}{3}}\right)\] by squaring and considering the off-diagonal terms.  Observe that these easily-derived estimates correspond to the limiting behavior of the transition regions in either of the more general theorems stated above.  

At the transition phase (i.e. on the blowup) the asymptotic behavior in both theorems depends crucially on the automorphic nature of the modular form $f$ or the Jacobi symbol.  Conrey, Farmer and Soundararajan's proof is based on the Poisson summation formula, see section 3 of their paper \cite{nondifferentiable}, which is the essential ingredient in proving the functional equations of the Dirichlet $L$-functions $L(s,\chi)$, i.e. $GL_1$ automorphy.  Likewise, one might consider Lemma 1 of section 3 of this paper to be the essential step which brings out the behavior of shifted convolution sums on the blowup.  The crucial step in the proof of Lemma 1 is the application of the functional equation of the $L$-function $L(s,f \times f)$, which follows from the automorphy of $f$.  

The Theorem proved by Conrey, Farmer and Soundararajan has one very surprising feature which is not apparent in our work.  The function $C(\alpha)$ appearing in their result is once continuously differentiable everywhere, but it is twice differentiable at $\alpha \Q $ if and only if $\alpha = 2 p/q$ with $p$ and $q$ both odd, see section 6 of their paper.  It is necessary that the sum $S(X,Y)$ has a sharp cut-off for $C(\alpha)$ to have such strange differentiability properties.  In our work, we do not prove an asymptotic result for a sharp cut-off, but we do have some flexibility in our choice of cut off functions.   In section 3, we see that the asymptotic we obtain on the blowup has a main term which depends on the particular shape of the cut-off function we choose.  This is unusual in analytic number theory.  It remains possible that averages of shifted convolution sums in the shift aspect when considered with a sharp cut-off also have strange differentiability properties at the transition region.  

The most essential connection, however, is that both theorems arise from averages of Fourier coefficients of Eisenstein series, and the transition regions can be understood by studying the resulting multiple Dirichlet series.  We first discuss the case of characters.  We have that \[S(X,Y) =  \frac{1}{(2 \pi i)^2} \int_{(c)}\int_{(c)} Z(s,w) \frac{X^s}{s}\frac{Y^w}{w} \,ds\,dw\] where $c>1$ and \[Z(s,w) := \sum_{\substack{m,n \geq 1 \\ m,n \text{ odd}}} \frac{\left(\frac{m}{n}\right)}{n^s m^w}.\] This is perhaps the first example of a multiple Dirichlet series, and information about the analytic properties of $Z(s,w)$ would determine the asymptotic behavior of $S(X,Y)$.  Goldfeld and Hoffstein in \cite{GHMDS} derive the analytic properties of $Z(s,w)$, which crucially follow from studying weight $\frac{1}{2}$ Eisenstein series for the congruence subgroup $\Gamma_0(4)$.  More precisely, they study the modified multiple Dirichlet series \[Z_+(s,w) = \sum_{\substack{ m \geq 1 \\ m \text{ squarefree}}} \frac{L(s,\chi_m)}{ m^{w}},\] where \[\chi_m(n) := \begin{cases}  \left(\frac{m}{n}\right) & m \equiv 1 \pmod 4 \\ \left(\frac{4m}{n}\right) & m \equiv 2,3 \pmod 4 \end{cases} \] and find that it has poles along $w=1$ and $w=\frac{3}{2}-s$.  In fact, if $E_{\frac{1}{2}}(z,s)$ is the weight $\frac{1}{2}$ Eisenstein series at the cusp 0 (see Goldfeld and Hoffstein for details), it has the Fourier expansion \[E_{\frac{1}{2}}(z,s) =  \sum_{m \geq 1} a_m(s,y)e(nx)\] where \[a_m(s,y) = \frac{L(2s,\chi_m)(1-\chi_m(2)2^{-2s})}{\zeta(4s)(1-2^{-4s})} \frac{y^s}{4^s}  K_m(s,y)\] is essentially $L(2s,\chi_m)$ times a $K$-Bessel function.  One finds that \[ \sum_{\substack{ m \leq Y \\ m \text{ odd}}} \frac{1}{2 \pi i} \int_{(\frac{1}{2})} a_m\left(s, 1/X\right)\,ds\] evaluates to a double mellin inverse of $Z(s,w)$ with some additional factors.  Ignoring covergence issues, or after sufficient smoothing, one can see how the transition region arises directly from $Z(s,w)$.  For simplicity, we work with $Z_+(s,w)$.  We have \begin{eqnarray*} \sum_{\substack{m \leq Y \\ m \text{ squarefree} }} \sum_{n \leq X} \chi_m(n) & = &  \frac{1}{(2\pi i )^2} \int_{(c)}\int_{(c)} Z_+(s,w) \frac{X^s}{s}\frac{Y^w}{w}\,ds\,dw \\ & = & \frac{1}{2\pi i} \int_{(c)} c(w) X\frac{Y^w}{w} + c_+^*(w)  \frac{X^\frac{3}{2} }{(\frac{3}{2}-w)w} \left(\frac{Y}{X}\right)^w \,dw \\
& & + \frac{1}{(2\pi i)^2}\int_{(c)} \int_{(\frac{3}{2} - \real(w) - \varepsilon)} Z_+(s,w)\frac{X^s}{s} \frac{Y^w}{w}\,ds \,dw \end{eqnarray*} where $c(w)$ and $c_+^*(w)$ are the meromorphic functions arising in Theorem 1 of \cite{GHMDS}.  They have poles at $w=\frac{1}{2}$, with $\res_{w=\frac{1}{2}} c(w) = - \res_{w=\frac{1}{2}} c_+^*(w)$ because of the intersection of the two singular divisors at $(s,w) = (\frac{1}{2},1)$.  Hence, if $Y/X\rightarrow \infty$, the terms involving $c(w)$ and $c_+^*(w)$ cancel, and if $Y/X \rightarrow 0$, the term involving $c(w)$ becomes a main term, and the one involving $c_+^*(w)$ becomes an error term of size $O(Y^\frac{3}{2})$.  If $Y/X$ remains constant, the term involving $c_+^*(w)$ gives the behavior on the blow-up.  

Our result on the following pages can be seen to arise from an identical situation.  Sums of the form \eqref{wavy} are essentially an average of Fourier coefficients of $|f|^2$.  Our method to treat these sums is to take a spectral expansion into Eisenstein series and Maass cusp forms, so one is led to compute a sum similar to \[\sum_{m \leq Y} \frac{1}{2 \pi i} \int_{(\frac{1}{2})} b_m\left(s,1/X \right)\,ds,\] where \[E(z,s) := \sum_{\gamma \in \Gamma_\infty \backslash \Gamma} \imag(\gamma z)^s = \sum_n b_n(s,y) e(nx) \] is the standard weight 0 real analytic Eisenstein series.  The integral along $\real(s) = \frac{1}{2}$ here arises from averaging over the continuous spectrum in our decomposition.  In this context, the analouge of $Z(s,w)$ is \[ \sum_h \sum_{ab=h} \left( \frac{a}{b}\right)^{s-\frac{1}{2}} \frac{1}{h^w} = \zeta\left(w-s+\frac{1}{2}\right)\zeta\left(w+s-\frac{1}{2}\right).\]  While not a multiple Dirichlet series, it is still a zeta function in two variables and has singular divisors at $w=s+\frac{1}{2}$ and $w= \frac{3}{2}-s$, the intersection of which gives rise to a transition behavior exactly as discussed above.  This will be carried out in explicit detail in the remainder of the paper.

Finally, one should note the connection between Riemann's nondifferentiable function and theta functions on the real axis, see the interesting paper of Duistermaat \cite{RiemannNondiffereniable}.

\section{Preliminaries}

Our main approach to the Main Theorem is to take the Petersson inner product of $y^k |f|^2$ against an incomplete Poincar\'{e} series $P_h(\cdot|\psi)$.  This approach was first introduced by Selberg \cite{SelbergSCS}, and has been successfully used by many other authors to study shifted convolution sums in the past (for an overview, see \cite{MichelLfcnsinfamilies}).  Throughout this paper we set $\Gamma = SL_2(\Z)$, and let $\mathcal{H}$ denote the upper half plane with its hyperbolic metric.  We work in the Hilbert space $L^2(\Gamma \backslash \mathcal{H})$ of square integrable measurable functions with the Petersson inner product \[\langle u,v\rangle = \int_{\Gamma \backslash  \mathcal{H}} u(z) \overline{v(z)}\,d_\mu z.\]  The symmetric operator \[\Delta = -y^2 \left( \frac{\partial^2}{\partial x^2}+ \frac{\partial^2}{\partial y^2} \right) \] acts on the subspace of smooth functions and moreover has a unique self-adjoint extension to all $L^2(\Gamma \backslash \mathcal{H})$, see Iwaniec \cite{SpecIw} chapter 4.  Given a classical holomorphic normalized cuspidal eigenform $f$ of weight $k$, set $F(z) := y^{\frac{k}{2}}f(z).$  We have that $|F| \in L^2(\Gamma \backslash \mathcal{H}),$ but on the other hand it is no longer holomorphic.  Let $\psi(y)$ be an infinitely differentiable compactly supported function on $\R_{>0}$, and let $\Gamma_\infty$ denote the stabilizer in $\Gamma$ of the cusp at infinity.  Then we define the incomplete Poincar\'{e} series \[P_h(z|\psi) := \sum_{\gamma \in \Gamma_\infty \backslash \Gamma} e(h\gamma z) \psi \left(\text{Im}(\gamma z)\right),\] which is a smooth and bounded function on $\Gamma \backslash \mathcal{H}$.  By unfolding the inner product \[\langle F  P_h(\cdot | \psi),F\rangle = \int_{\Gamma \backslash \mathcal{H}} y^k|f(z)|^2 P_h(z|\psi) \,d_\mu z\] on the Poincar\'{e} series, we find that \[ \langle F P_h(\cdot | \psi),F\rangle  =  \sum_{n =1}^\infty \lambda_f(n)\lambda_f(n+h) (n(n+h))^{\frac{k-1}{2}}\int_0^\infty \psi(y) e^{-4 \pi (n+h)y} y^{k-2} \, dy,\] i.e. this inner product is a smoothed shifted convolution sum with cut-off function given in terms of an integral transform (similar to the Laplace transform) of $\psi$.  The behavior of $\psi(y)$ as $y$ tends to 0 is crucial to control the length of the shifted convolution sum.  In connection with the previous section, it should be noted that if $\psi$ were a delta function, then taking the inner product against $P_h(z|\psi)$ is equivalent to taking the $h$-th Fourier coefficient.

Let $u_j$ be a complete orthonormal system of cusp forms which are eigenfunctions of the Laplace operator and all Hecke operators.  Because we are only working in level 1, we need not worry about old forms or the Hecke operators whose index divides the level.  Define the real analytic Eisenstein series by \[E(z,s) := \sum_{\gamma \in \Gamma_\infty \backslash \Gamma } \text{Im}(\gamma z)^s\] for $\real(s) > 1$, and in general by analytic continuation.  For each $s \neq 0,1$, the Eisenstein series are also eigenfunctions for the Laplace operator and all the Hecke operators.  We have then that \[\sum_{j = 1}^\infty \langle P_h(\cdot |\psi), u_j\rangle u_j(z) + \frac{1}{4\pi} \int_{-\infty}^\infty \langle P_h(\cdot|\psi),E(\cdot, 1/2 + it)\rangle E(z, 1/2 + it) \,dt\] converges to $P_h(z|\psi)$ in the norm topology on $L^2(\Gamma \backslash \mathcal{H})$, see Theorems 4.7 and 7.3 in \cite{SpecIw}.  Then we have that \begin{equation}\label{expansion}\langle F P_h(\cdot | \psi),F\rangle =  \sum_{j=1}^\infty \langle P_h(\cdot | \psi), u_j\rangle \langle Fu_j,F\rangle + \frac{1}{4\pi} \int_{-\infty}^\infty \langle P_h(\cdot | \psi) , E(\cdot, 1/2+ it) \rangle \langle F  E(\cdot, 1/2+ it), F \rangle\,dt,\end{equation} and we will see later that the convergence is absolute and uniform in $h$.  


If $u \in L^2\left(\Gamma \backslash \mathcal{H}\right)$ and $\Delta u = \lambda u$ with $\lambda = s(1-s),$ then $u(z)$ has a Fourier expansion given in terms of the $K$-Bessel function $K_\nu(z)$,  which is an exponentially decaying solution to the differential equation \[z^2 f'' +z f' - \left(z^2 + \nu^2\right) f = 0.\]  
We will primarily be interested in the $K$-Bessel function for purely imaginary $\nu$, and note some useful properties of these functions: first, that $K_{\nu}(z)$ is real for real $z$, second, that as $\imag(\nu) \rightarrow \infty$ in a fixed vertical strip, $K_\nu(z)$ is decaying exponentially, and last, that for $\nu = i t$, with $t \in \R$, $K_{it}(z)$ has a branch cut, which we take to be along the negative real axis in the $z$-plane.  As $z\rightarrow 0$, we have that \[|K_{it}(z)| \sim \pi  \left| \frac{\sin(t \log z/2)}{\Gamma(1+it) \sinh(\pi t)}\right|\] so long as one avoids the branch cut.

\section{Proof of Theorem}

We now proceed to the proof of the Main Theorem of the paper by summing the right side of  \eqref{expansion} over $h$.  Pointwise, the largest term comes from the discrete spectrum (see \cite{SarnakSCS}), however, on average, the continuous spectrum dominates.  We start with the continuous spectrum contribution to \eqref{expansion}.

\subsection{Eisenstein Series}

The Eisenstein series $E(z,s)$ has a Fourier expansion given by

\[E(z,s) = \varphi(0,s)+\sum_{n \neq 0} \varphi(n,s) W_{s}(nz), \]
where
\[W_{s}(z) = 2y^{\frac{1}{2}} K_{s-\frac{1}{2}}(2 \pi y) e(x).\] If $n= 0$ and $s \neq \frac{1}{2}$  then \[\varphi(0,s) = y^{s} + \frac{\xi(2s-1)}{\xi(2s)}y^{1-s},\] and if $n \neq 0$, the $n$-th Fourier coefficient of $E(z,s)$ is given by
\[\varphi(n,s) = \xi(2s)^{-1} |n|^{-1/2} \sum_{ab = |n|} \left(\frac{a}{b}\right)^{s-1/2},\] where $\xi(s) = \pi^{-s/2}\Gamma(s/2)\zeta(s)$ denotes the completed Riemann zeta function which has the functional equation $\xi(s) = \xi(1-s)$. If $s=\frac{1}{2}+it$, we find by unfolding the Poincar\'{e} series that the inner product \[\langle P_h(\cdot | \psi), E(\cdot, 1/2+it)\rangle =  \frac{2}{\xi(2s-1)}  \sum_{ab = h }\left(\frac{a}{b}\right)^{s-\frac{1}{2}} \int_0^\infty \frac{\psi(y)}{y^{3/2}}  e^{-2 \pi h y} K_{s-\frac{1}{2}}(2 \pi h y) \,dy \] where we have used that $\overline{\xi(2s)} = \xi(2 \overline{s}) = \xi(2-2s) = \xi(2s-1)$. We can also unfold the second inner product on the Eisenstein series.  Following Iwaniec \cite{ClassIw} chapter 13, set
\[L(s,f\times f) := \zeta(s) L(s,\sym^2 f) = \zeta(2s)L(s,f \otimes f) = \zeta(2s) \sum_{n\geq 1} \frac{\lambda_f(n)^2}{n^s},\] where the last equality is valid only for $\mathop{\rm Re}(s) >1$.  This $L$-function admits the functional equation \[\Lambda(s, f \times f) = \Lambda(1-s, f \times f),\] where \[\Lambda(s, f \times f) = L_\infty(s, f \times f)L(s, f \times f)\] with \[L_\infty(s,f \times f) := (2 \pi)^{-2s} \Gamma(s)\Gamma(s+k-1) .\] By unfolding when $\real(s) >1$, and in general after analytic continuation, 
\[ \langle F E(\cdot, s),F\rangle = \frac{\Lambda(s, f \times f)}{(4 \pi)^{k-1} \xi(2s)}.\]

Going back to the spectral expansion \eqref{expansion} and pulling these two inner products together, we have that the Eisenstein series contribution to $\langle F P_h(\cdot | \psi),F\rangle$ is  
\[E_{f,h}(\psi) :=  \frac{1}{(4 \pi)^{k-1}} \frac{1}{2 \pi i} \int_{(\frac{1}{2})} \frac{\Lambda(s, f \times f)}{\xi(2s)\xi(2s-1)}\sum_{ab = h }\left(\frac{a}{b}\right)^{s-\frac{1}{2}}  \int_0^\infty \frac{\psi(y)}{y^{3/2}}  e^{-2 \pi h y} K_{s-\frac{1}{2}}(2 \pi h y) \,dy \, ds.\]
The completed $L$ and zeta functions have the same number of gamma factors in the numerator and denominator, and Bessel function decays rapidly as $|\imag(s)| \rightarrow \infty$, so the contour integral converges rapidly.  It is interesting to note that the $s$ integral only makes sense because $\xi(s)^{-1}$ has no poles and at most polynomial growth on the $\real(s) =1$ line, i.e. due to the prime number theorem.  Indeed, one sees that the contour is constrained between the poles in the critical strips of the two $\xi(s)^{-1}$ functions appearing here.   Due to the mysterious nature of the residues at these poles, shifting contours seems to be a futile approach to understanding the asymptotic size of $E_{f,h}(\psi)$.

Introducing a sum over $h$ clears this obstruction.  We have to compute
\begin{equation}\label{eisbeginning}\sum_{h \leq Y} E_{f,h}(\psi) = \frac{1}{(4 \pi)^{k-1}} \int_0^\infty \frac{\psi(y)}{y^{\frac{3}{2}}} \frac{1}{2 \pi i} \int_{(\frac{1}{2})} \frac{\Lambda(s, f \times f)}{\xi(2s) \xi(2s-1)} \sum_{h \leq Y} \sum_{ab = h} \left(\frac{a}{b} \right)^{s-\frac{1}{2}} e^{-2 \pi y h} K_{s-\frac{1}{2}}(2 \pi y h)\,ds\,dy,\end{equation} and can evaluate the sum over $h$ by standard techniques.  We have by adapting Theorem 12.4 from \cite{TOTRZF} (due to Van Der Corput) that 
\[\sum_{h \leq x} \sum_{ab = h} \left( \frac{a}{b}\right)^{s-\frac{1}{2}} = \zeta(2s) \frac{x^{s+\frac{1}{2}}}{s+\frac{1}{2}} + \zeta(2-2s) \frac{x^{-s+\frac{3}{2}}}{-s+ \frac{3}{2}} + O_{s,\varepsilon}\left(x^{\frac{27}{82}+|\real(s) -\frac{1}{2}| +\varepsilon}\right),\] where the implied constants depend at most polynomially on $|s|$.   It should be noted that the error term here is not the best currently known, however, it is sufficiently small as to not contribute to the final result of this paper.  By partial summation, \begin{eqnarray*} & & \sum_{h \leq Y} \sum_{ab = h} \left(\frac{a}{b} \right)^{s-\frac{1}{2}} e^{-2 \pi y h} K_{s-\frac{1}{2}}(2 \pi y h) \\ & = & \zeta(2s) \int_0^Y u^{s-\frac{1}{2}} e^{-2\pi y u } K_{s-\frac{1}{2}} (2 \pi y u )\, du  + \zeta(2-2s) \int_0^Y u^{\frac{1}{2}-s} e^{-2\pi y u } K_{s-\frac{1}{2}} (2 \pi y u )\, du\\
& & + \,\,\, O_{s,\varepsilon}\left(e^{-2 \pi y Y}K_{s-\frac{1}{2}}(2 \pi y Y) Y^{\frac{27}{82}+|\real(s) -\frac{1}{2}|  +\varepsilon}\right). \end{eqnarray*} Note that the two integrals appearing in the displayed equation are interchanged under the transformation $s \longleftrightarrow 1-s$, by symmetry of the the Bessel function.  The contour integral over $s$ in $\sum E_{f,h}(\psi)$ is also symmetric under $s \longleftrightarrow 1-s$ so that these two integrals are identical in the overall sum, and we need only work one of them out.  The first integral can be evaluated explicitly, and the answer will be in terms of hypergeometric functions.  The confluent hypergeometric function that will appear below is defined by the power series \[{}_1F_1(a,b,z) = \sum_{n=0}^\infty \frac {a^{(n)} z^n} {b^{(n)} n!} \] where \[a^{(n)}=a(a+1)(a+2)\cdots(a+n-1) = \frac{\Gamma(a+n)}{\Gamma(a)}\] is called either the `rising factorial' or `Pochhammer symbol'. 
The theory of hypergeometric functions is developed in detail in \cite{HTFvol1}.  It is entire on $\C$ separately in each variable except for simple poles at $b=0,-1,-2,\ldots$ by the absolute and uniform convergence of the defining series, hence it is meromorphic on $\C^3$.  We have that the residues at these poles are given by  \[\res_{b=-n} {}_1F_1(a,b,z) = \frac{\Gamma(a+n+1)(-1)^n}{\Gamma(a)\Gamma(n+2)\Gamma(n+1)} z^{n+1} {}_1F_1(a+n+1,n+2,z),\] see Gradshteyn and Ryzhik \cite{GR}, 9.214.  

We have the formulae \[K_\nu(z) = \frac{\pi}{2} \frac{i^{\nu}J_{-\nu}(iz) - i^{-\nu} J_{\nu}(iz)}{\sin \pi \nu},\] \[e^{-z} \frac{\Gamma\left( s+ \frac{1}{2}\right)}{\left(\frac{iz}{2}\right)^{s-\frac{1}{2}}}J_{s-\frac{1}{2}}(iz) = {}_1F_1(s,2s,-2z),\] \[ \frac{d}{dz} {}_1F_1(a,b,z) = \frac{a}{b}{}_1F_1(a+1,b+1,z),\] and \[ \frac{d}{dz} \left(z^{b-1}{}_1F_1(a,b,z)\right) = (b-1)z^{b-2}{}_1F_1(a,b-1,z),\] where $J_\nu(x)$ is the $J$-Bessel function, the second formula can be found in \cite{GR}, 9.215 \#3 and the last two formulae can be found in \cite{HTFvol1} section 2.1.2.  From these it follows that \[ \int_0^Y u^{s-\frac{1}{2}} e^{-2\pi y u } K_{s-\frac{1}{2}} (2 \pi y u )\, du = \frac{1}{4 (\pi y)^\frac{1}{2} s} \left(A_f(y,s)+ B_f(y,Y,s)+C_f(y,Y,s)\right),\] where \[A_f = A_f(y,s) = (\pi y)^{-s} \Gamma(1/2 +s )\] \[B_f = B_f(y,Y,s) = (\pi y)^{-s} \Gamma(1/2+s) {}_1F_1(-s,1-2s, -4 \pi y Y) \] and \[C_f = C_f(y,Y,s) =  (\pi y Y^2)^{s} \Gamma(1/2 -s) {}_1F_1(s, 1+2s, -4 \pi yY). \]

The $\zeta(2s)$ from the evaluation of the sum over $h$ cancels against the $\zeta(2s)$ in the denominator of \eqref{eisbeginning}, eliminating its poles.  Given that $A_f+B_f+C_f$ is holomorphic in $s$ past the $\real(s) = 0$ line, we are free to pass the contour to the left.  Picking up a residue at $s=0$ we get \begin{eqnarray*}
 \sum_{h \leq Y} E_{f,h}(\psi) & = &- \frac{\Gamma(k)}{(4 \pi)^k} \frac{L(1,\sym^2 f)}{2 \zeta(2)} \int_0^\infty \frac{\psi(y)}{y^2}\,dy \\
& & +\,\,\, \frac{1}{(4\pi)^{k-\frac{1}{2}}} \int_0^\infty \frac{\psi(y)}{y^2}\frac{1}{2 \pi i} \int_{(-a)}\frac{\Lambda(s, f\times f)}{\xi(2s-1)}\frac{\left(A_f(y,s) + B_f(y,Y,s) + C_f(y,Y,s)\right)}{\pi^{-s}\Gamma(s+1)}\,ds\,dy \\ & &  +\,\,\, O_{k,\varepsilon}\left(Y^{\frac{27}{82}+\varepsilon} \int_0^\infty \frac{|\psi(y)|}{y^{\frac{3}{2}}}\,dy\right),
\end{eqnarray*}
where $0<a<\frac{1}{2}$, and we have again made use of the prime number theorem in estimating the error term.

The integrals over $y$ converge at $\infty$, so the asymptotic size of $\sum E_{f,h}(\psi)$ depends only on the behavior of $\psi(y)$ for small $y$.  We now make some estimates assuming that $y$ is small. The remaining contour integral \[ \frac{1}{2 \pi i} \int_{(-a)}\frac{\Lambda(s, f\times f)}{\xi(2s-1)}\frac{\left(A_f(y,s) + B_f(y,Y,s) + C_f(y,Y,s)\right)}{\pi^{-s}\Gamma(s+1)}\,ds \] is a sum of three terms coming from $A_f$, $B_f$ and $C_f$.  The term coming from $A_f$ is $\ll_k y^{\frac{1}{2}}$ as $y \rightarrow 0$.  The hypergeometric function appearing in $B_f$ is bounded above and below by universal constants in the half-plane $\real(s)<0$ and when $0 \leq yY \leq 4 \pi$, thus term coming from $B_f$ is also $\ll y^{\frac{1}{2}}$, uniformly in $Y$.  Thus it remains to inspect the term coming from $C_f$.  Explicitly, let
\begin{eqnarray*} C_{f,1}(y,Y) & := &  \frac{1}{2 \pi i} \int_{(-a)}\frac{\Lambda(s, f\times f)}{\xi(2s-1)}\frac{C_f(y,Y,s)}{\pi^{-s}\Gamma(s+1)}\,ds \\ & = & \frac{1}{2 \pi i} \int_{(-a)}\frac{\Lambda(s, f\times f)}{\xi(2s-1)}\frac{\Gamma\left(\frac{1}{2} -s\right)}{\Gamma(s+1)} {}_1F_1(s, 1+2s, -4 \pi yY) \left(\pi^2 y Y^2\right)^{s} \,ds. \end{eqnarray*}
Thus \begin{eqnarray*} \sum_{h \leq Y} E_{f,h}(\psi)& = & - \frac{\Gamma(k)}{(4 \pi)^k} \frac{L(1,\sym^2 f)}{2 \zeta(2)} \int_0^\infty \frac{\psi(y)}{y^2}\,dy +  \frac{1}{(4\pi)^{k-\frac{1}{2}}} \int_0^\infty \frac{\psi(y)}{y^2}C_{f,1}(y,Y)\,dy \\ & & + \,\,\,O_{k,\varepsilon}\left(Y^{\frac{27}{82}+\varepsilon} \int_0^\infty \frac{|\psi(y)|}{y^{\frac{3}{2}}}\,dy\right).\end{eqnarray*} In the transition region, $C_{f,1}(y,Y)$ is the crucial term.

\begin{lemma}
Suppose that $y,Y \in \R_{>0}$ with $Y$ becoming large and $y$ becoming small.  Then
\[ C_{f,1}(y,Y) = \frac{1}{2 \sqrt{\pi}}c_f\left(4 \pi yY^2\right) + O_k\left(y^{\frac{1}{2}}\right), \]  where \[ c_f(\alpha) = =\frac{\pi^\frac{3}{2}}{2}\alpha \sum_{n \geq 1} \lambda_f(n)^2 W_k(\pi^2 n \alpha), \] and \[W_k(x) =  \frac{1}{2 \pi i}\int_{(1+a)} \frac{\Gamma(s+k-1) \Gamma(s-\frac{1}{2})}{\Gamma(2-s)}x^{-s}  \,ds\]for any fixed $a>0$. 
\end{lemma}

\begin{proof}
We apply the functional equation for the $L$-function to find that \[C_{f,1}(y,Y) =  \pi^2 y Y^2 \frac{1}{2 \pi i}\int_{(1+a)}\frac{L(s,f\times f)}{\zeta(2s)} \frac{\Gamma(s+k-1) \Gamma(s-\frac{1}{2})}{\Gamma(2-s)} {}_1F_1(1-s,3-2s,-4 \pi y Y) \left(4 \pi^3 yY^2\right)^{-s}\,ds.\]  From the definition one finds that ${}_1F_1(1-s,3-2s,u) = 1 + O_s(u),$ which we use to eliminate the hypergeometric function from the above expression.  We proceed in two slightly different ways depending on whether $yY^2$ becomes large or becomes small.  If $yY^2$ remains bounded, either approach is acceptable.  First, assume that $yY^2$ is becoming small.  In this case, choose $a=\frac{1}{4}$, and observe that the $s$-dependence in the hypergeometric function is uniformly bounded along the line $\real(s) = \frac{5}{4}$.  Together with the rapid decay of the integrand of $C_{f,1}(y,Y)$, this gives us that \[ \left|C_{f,1}(y,Y) - \frac{1}{2 \sqrt{\pi}}c_f\left(4 \pi yY^2\right)\right|  \ll_k  y^{\frac{1}{2}}(yY^2)^{\frac{1}{4}}.\] 

If $yY^2$ becomes large shift the line of integration to the right, past the pole of the hypergeometric function at $s=\frac{3}{2}$ to $a = \frac{3}{4}$.  The contribution to $C_{f,1}(y,Y)$ coming from this residue is $\ll_k y^{\frac{1}{2}}$, uniformly in $Y$, and the $s$-dependence in the hypergeometric function is uniformly bounded along the line $\real(s) = \frac{7}{4}$.  We find in this case that \[\left|C_{f,1}(y,Y) - \frac{1}{2 \sqrt{\pi}}c_f\left(4 \pi yY^2\right)\right| \ll_k y^{\frac{1}{2}}\left( 1+ (yY^2)^{-\frac{1}{4}}\right). \]  
In either case, we obtain the error term stated in the Lemma.

\end{proof}
\begin{lemma}
The function $c_f(\alpha)$ defined above is $C^\infty(\R_{>0})$.  As $\alpha \rightarrow \infty$ it decays faster than any polynomial, and as $\alpha \rightarrow 0$
\[c_f(\alpha) = \frac{\Gamma(k)L(1,\sym^2 f)}{2 \zeta(2)} + E_k(\alpha)\] where $E_k(\alpha) = o_k\left(\alpha^{\frac{1}{2}}\right)$ unconditionally, and $E_k(\alpha) = O_{k,\varepsilon}\left(\alpha^{\frac{3}{4}-\varepsilon}\right)$ for any $\varepsilon>0$ assuming the Riemann hypothesis for the classical Riemann zeta function.
\end{lemma}
\begin{proof}
The $W_k(x)$ defined above is $C^\infty(\R)$ and its integrand has no poles to the right, thus $W_k(x)$ is rapidly decaying as $x \rightarrow + \infty$.  After differentiating the series for $c_f(\alpha)$ in the second line of Lemma 1 arbitrarily many times, the resulting series for $c^{(n)}_f(\alpha)$ converges absolutely for any $\alpha >0$, so $c_f \in C^\infty(\R_{>0})$.  The rapid decay of $c_f(\alpha)$ follows from that of $W_k(x)$.  By shifting the line of integration in the definition of $c_f(\alpha)$ to the left, we investigate the behavior of $c_f(\alpha)$ near $\alpha = 0$.  The main term comes from the residue of the pole of $L(s, f \times f)$ at $s=1$, and the error term is estimated by pushing the contour just past the $\real(s) = \frac{1}{2}$ line, or to the $\real(s) = \frac{1}{4}+\varepsilon$ line if one assumes the Riemann hypothesis.
\end{proof}

Thus we have \begin{prop} For $E_{f,h}(\psi)$ and $c_f(\alpha)$ as defined above, we have that \[ \sum_{h \leq Y} E_{f,h}(\psi) = \frac{1}{(4 \pi)^k} \int_0^\infty \left( c_f\left(4 \pi yY^2\right) - \frac{\Gamma(k) L(1,\sym^2 f)}{2\zeta(2)} \right) \frac{\psi(y)}{y^2}\,dy + O_{k,\varepsilon}\left(Y^{\frac{27}{82}+\varepsilon} \int_0^\infty \frac{\psi(y)}{y^{\frac{3}{2}}}\,dy\right).\] \end{prop}

The term involving $\frac{27}{82}$ is smaller than the remainder terms coming from Maass forms, as we will see in the next section.

\subsection{Maass Forms}

The discrete spectrum of $\Delta$ is spanned by Maass cusp forms.  The Hecke algebra acting on $L^2( \Gamma \backslash \mathcal{H})$ is defined to be the algebra generated by the commuting self-adjoint bounded operators $T_n$, where for $u \in L^2(\Gamma \backslash \mathcal{H})$, define $T_n$ by \[(T_nu)(z) := \frac{1}{\sqrt{n}} \sum_{ad=n} \sum_{b \!\!\!\!\! \pmod d} u\left( \frac{az+b}{d}\right).\]  These operators commute with $\Delta$ as well, so in fact our basis of Maass forms can be taken to be eigenfunctions of the Hecke algebra as well, and we denote the Hecke eigenvalues of the Maass form $u_j$ by $\lambda_{u_j}(n)$.  A Maass form of Laplace eigenvalue $\lambda_j = s_j(1-s_j) = \frac{1}{4} + t_j^2$ is cuspidal, so it has a Fourier expansion of the form
\[u_j(z) = \sum_{n \neq 0} a_{u_j}(n) W_{s_j}(nz), \]
where
\[W_{s_j}(z) = 2y^{\frac{1}{2}} K_{it_j}(2 \pi y) e(x).\]

For $\Gamma = SL_2(\Z)$, it was known to Selberg in the early 50s that the smallest Laplace eigenvalue $\lambda_1$ is $>\frac{1}{4}$, hence $t_j \in \R$.  For a proof of this fact, see \cite{HejhalSTFvol2}, chapter 11.  Computationally, it has been verified that $t_1 = 9.53369526\ldots$, see for example \cite{Hejhalnumerical}.  To apply the spectral theorem we must assume the normalization $||u_j||_{L^2}^2 = 1$, in which case the Fourier coefficient and Hecke eigenvalue are related by  \[a_{u_j}(n) = \left(\frac{\cosh\pi t_j}{2 |n| L(1,\sym^2 u_j)}\right)^{\frac{1}{2}}\lambda_{u_j}(n),\] where the symmetric square $L$-function appearing here is defined \[ L(s, \sym^2 u_j) = \prod_p \left(1-\frac{\alpha_{u_j}(p)^2}{p^s}\right)^{-1}  \left(1-\frac{\alpha_{u_j}(p)\beta_{u_j}(p)}{p^s}\right)^{-1}  \left(1-\frac{\beta_{u_j}(p)^2}{p^s}\right)^{-1}.\]  The Hecke eigenvalues conform to the bound $|\lambda_{u_j}(n)| \leq d(n)n^{\theta}$, where the generalized Ramanujan conjecture implies that $\theta = 0$ is admissible, and the best known unconditional bound is due to Kim and Sarnak \cite{KimSarnak}, which gives $\theta = \frac{7}{64}$.

We call the Maass form contribution to \eqref{expansion} \[M_{f,h}(\psi) := \sum_{j=1}^\infty \langle P_h(\cdot | \psi),u_j\rangle \langle F u_j,F\rangle.\] By unfolding we have \[ \langle P_h(\cdot | \psi), u_j\rangle  =  \left(\frac{\cosh \pi t_j}{L(1, \sym^2 u_j)}\right)^\frac{1}{2} \lambda_{u_j}(h) \int_0^\infty \frac{\psi(y)}{y^{\frac{3}{2}}} e^{-2\pi h y} K_{it_j}(2 \pi h y)\,dy \] so that \[ \sum_{h \leq Y} M_{f,h}(\psi) = \int_0^\infty  \frac{\psi(y)}{y^{\frac{3}{2}}} \sum_{j=1}^\infty \left(\frac{\cosh \pi t_j}{L(1, \sym^2 u_j)}\right)^\frac{1}{2} \langle F u_j,F\rangle \sum_{h \leq Y} \lambda_{u_j}(h)  e^{-2\pi h y} K_{it_j}(2 \pi h y)\,dy \]

\begin{lemma}
The spectral sum \[\sum_{j=1}^\infty \left(\frac{\cosh \pi t_j}{L(1, \sym^2 u_j)}\right)^\frac{1}{2} \langle F u_j,F\rangle \sum_{h \leq Y} \lambda_{u_j}(h)  e^{-2\pi h y} K_{it_j}(2 \pi h y) \]
appearing in $\sum M_{f,h}(\psi)$ converges absolutely.
\end{lemma}
\begin{proof}
There are three factors in the summand: that which involves $\cosh \pi t_j$ , the inner product, and the sum over $h$.  We show that the first of these two balance each other, and then show that the sum over $h$ decays rapidly in $|t_j|$, uniformly in the other variables.  To study $|\langle F u_j, F\rangle|$ we will use a beautiful formula of Watson \cite{Watson}, but follow a classical work-out of it from Soundararajan's paper \cite{WeakSC}.  In that paper both $f$ and $u_j$ are normalized to have mass 1, but in this paper we take $f$ to be Hecke normalized so that Watson's formula is

\[|\langle F u_j, F \rangle |^2 = \frac{1}{8} \left( \frac{\Gamma(k)}{(4 \pi)^{k}} \frac{\text{Vol}\left(\Gamma \backslash \mathcal{H}\right)}{\zeta(2)}\right)^{2} \frac{\Lambda(\frac{1}{2}, f \times f \times u_j)}{L_\infty(1,\sym^2 f)^2 \Lambda(1, \sym^2 u_j)},\] where
\[\Lambda(s,f \times f \times u_j) = L_\infty(s,f \times f \times u_j) L(s, f \times f \times u_j),\]
\[L_\infty(s, \sym^2 f ) =  \pi^{-\frac{3}{2}s}\Gamma\left(\frac{s+1}{2}\right)\Gamma\left(\frac{s+k-1}{2}\right)\Gamma\left(\frac{s+k}{2}\right),\]
\[ \Lambda(s,\sym^2 u_j) = L_\infty(s,\sym^2 u_j) L(s,\sym^2 u_j),\]
\[L_\infty(s, f \times f \times u_j) = \pi^{-4s}\prod_{\pm} \Gamma\left(\frac{s+k-1\pm i t_j}{2}\right)\Gamma\left(\frac{s+k \pm it_j}{2}\right)\Gamma\left(\frac{s+1\pm i t_j}{2}\right)\Gamma \left(\frac{s \pm it_j}{2}\right) ,\]
\[L_\infty(s,\sym^2 u_j) = \pi^{-3s/2}\Gamma\left(\frac{s-2it}{2}\right)\Gamma\left(\frac{s}{2}\right) \Gamma\left(\frac{s+2it}{2}\right)\]
and
\[ L(s,f \times f \times u_j) = \prod_p \left(1 - \frac{\alpha_f(p)^2\alpha_{u_j}(p)}{p^s}\right)^{-1}\left(1 - \frac{\alpha_{u_j}(p)}{p^s}\right)^{-2}\left(1 - \frac{\beta_f(p)^2\alpha_{u_j}(p)}{p^s}\right)^{-1}\]\[\times \left(1 - \frac{\alpha_f(p)^2\beta_{u_j}(p)}{p^s}\right)^{-1}\left(1 - \frac{\beta_{u_j}(p)}{p^s}\right)^{-2}\left(1 - \frac{\beta_f(p)^2\beta_{u_j}(p)}{p^s}\right)^{-1}.\] The archimedian parts evaluate to
\[ \frac{L_\infty(\frac{1}{2}, f \times f \times u_j)}{L_\infty(1,\sym^2 f)^2L_\infty(1,\sym^2 u_j)} 
= 4 \pi^2\frac{ |\Gamma(k-\frac{1}{2}+ i t_j)|^2}{ \Gamma(k)^2} \]
after repeated application of the duplication formula.  Using this, Watson's formula simplifies to \begin{equation}\label{Watson2}| \langle F u_j,F\rangle | = \sqrt{2} \frac{|\Gamma(k-\frac{1}{2} + it_j)|}{(4 \pi)^k} \left(\frac{L(\frac{1}{2}, f \times f \times u_j)}{L(1,\sym^2 u_j)}\right)^{\frac{1}{2}}\end{equation} for $f$ Hecke normalized, and $u_j$ mass 1 normalized.  By applying Stirling's formula, we find that $|\Gamma(k-\frac{1}{2} + it_j)|(\cosh \pi t_j)^\frac{1}{2}$ is polynomially bounded as $|t_j| \rightarrow \infty$.  Together with standard convexity bounds in the $|t_j|$ aspect for $L(\frac{1}{2},f \times f \times u_j)$, we find that \[ \left|\frac{\cosh \pi t_j}{L(1, \sym^2 u_j)}\right|^\frac{1}{2} |\langle F u_j,F\rangle| \] is polynomially bounded as $|t_j|$ gets large.  

Now we turn to the sum over $h$.  We have by a ``folklore'' result written down by Hafner and Ivi\'{c} \cite{HafnerIvic} that \[ \sum_{h \leq Y} \lambda_{u_j}(h) \ll_{u_j} Y^{\frac{1}{3}(1+ \theta)}\] where $\theta = 0$ or $=\frac{7}{64}$ as above, and where the implied constants depend at most polynomially on $|t_j|$.  As on the Eisenstein series side, the conjectural truth is $O_{u,\varepsilon}(Y^{\frac{1}{4}+\varepsilon})$, but this seems very difficult.  By partial summation \[ \sum_{h \leq Y} \lambda_u(h)  e^{-2\pi h y} K_{it_j}(2 \pi h y) \ll_{u_j}  e^{-2\pi y Y} K_{it_j}(2 \pi y Y) Y^{\frac{1}{3}(1+ \theta)} +  \int_{\frac{1}{2}}^Y u^{\frac{1}{3}(1+ \theta)} \left| \frac{\partial}{\partial u} e^{-2 \pi y u} K_{it_j}(2 \pi y u)\right| \,du, \] where the implied constants again depend at most polynomially on $|t_j|$.   We have that $|K_{it}(u)| \sim \pi \frac{|\sin(t \log u/2)|}{|\Gamma(1+it) \sinh(\pi t)|}$ for small $u$, thus after taking derivatives, changing variables, and using the power series expansion for the lower incomplete gamma function, we have that  \begin{equation}\label{maasssum} \sum_{h \leq Y} \lambda_u(h)  e^{-2\pi h y} K_{it_j}(2 \pi h y) \leq P(t_j) e^{-\frac{\pi}{2} |t_j|} Y^{\frac{1}{3}(1+\theta)},\end{equation} uniformly in $y$, where $P(t)$ is a real-valued function on $\R$ that grows at most polynomially as $|t|$ becomes large.  From these estimates together with Weyl's law (see, e.g. Iwaniec \cite{SpecIw}) \[ \sum_{|t_j|\leq T} 1 = \frac{\text{Vol}\left(\Gamma \backslash \mathcal{H}\right)}{4 \pi } T^2 + O_\varepsilon\left((1+T)^{1+\varepsilon}\right)\] the Lemma follows.
\end{proof}
We apply trivial estimates along with \eqref{maasssum} to obtain \begin{prop} For $M_{f,h}(\psi)$ defined above we have  \[\sum_{h \leq Y} M_{f,h}(\psi) \ll_k Y^{\frac{1}{3}(1+\theta)}\int_0^\infty \frac{|\psi(y)|}{y^{\frac{3}{2}}}\,dy. \] \end{prop} The application of trivial bounds is justified by the absolute convergence given by Lemma 3.

\section{The Main Theorem and its Corollaries}

Drawing together the propositions from the two preceding sections, we obtain
\begin{mt}
Let $\psi $ be any measurable function on $\R_{>0}$ such that the incomplete Poincar\'{e} series \[P_h(z|\psi) := \sum_{\gamma \in \Gamma_\infty \backslash \Gamma} e(h\gamma z) \psi(\imag(\gamma z))\] is a smooth and bounded $L^2$ function on $\Gamma \backslash \mathcal{H}$.  Denote the shifted convolution sum \[ S_{f}(\psi,Y) = \sum_{h \leq Y} \sum_{n \geq 1} \lambda_f(n)\lambda_f(n+h) (n(n+h))^{\frac{k-1}{2}}\int_0^\infty \psi(y) e^{-4 \pi (n+h)y} y^{k-2} \, dy.\]  Then we have that \[ S_f(\psi,Y) =  \frac{1}{(4 \pi)^k} \int_0^\infty \left( c_f\left(4 \pi yY^2\right) - \frac{\Gamma(k) L(1,\sym^2 f)}{2\zeta(2)} \right) \frac{\psi(y)}{y^2}\,dy+ O_k\left(Y^{\frac{1}{3}(1+\theta)}\int_0^\infty \frac{|\psi(y)|}{y^{\frac{3}{2}}}\,dy\right),\] where $c_f(\alpha)$ is the function defined in Lemma 1 and whose properties are given in Lemma 2, and $\theta =0$ or $\frac{7}{64}$ depending on whether we assume the truth of the generalized Ramanujan conjecture or not.
\end{mt}

It should be noted that the size of the exponent of $Y$ in the error term here depends on the sharp cut-off in $h$, and could be made smaller if we were to smooth that sum.  Now we make some choices for $\psi$, and record the results as corollaries.  First, we give the result stated in the introduction.
\begin{proof}[Proof of Corollary 1.]
Let $\psi$ be a smooth approximation to a point mass.  Specifically, let $\psi$ be smooth, non-negative, supported on a set of radius $X^{-4}$ about the point $y = \frac{1}{4 \pi X}$ and have mass $1$.  Then for any continuously differentiable function $\phi$ on $\R_{>0}$, we have \[\left|\phi\left(\frac{1}{4\pi X}\right) - \int_0^\infty \psi(y) \phi(y)\,dy \right| \ll \left|\phi'\left(\frac{1}{4\pi X}\right)\right| X^{-4},\] where the implied constants are absolute.  First, let $\phi(u) = u^{k-2}e^{-4\pi(n+h)u}$, so that $\left| \phi' \left(\frac{1}{4\pi X}\right)\right| \ll_k (n+h) X^{3-k} e^{-\frac{n+h}{X}}$, and thus we find that\[\left|S_f(\psi,Y) - \sum_{h \leq Y} \sum_{n \geq 1} \lambda_f(n) \lambda_f(n+h) \frac{(n(n+h))^{\frac{k-1}{2}}}{X^{k-2}} e^{-\frac{n+h}{X}}\right| \ll_{k,\varepsilon} X^{1+\varepsilon},\] hence the difference between the left hand sides of the Main Theorem and Corollary 1 is $\ll_\varepsilon X^\varepsilon$. Secondly, let \[\phi(u) = \left(c_f(4 \pi u Y^2) - \frac{\Gamma(k) L(1,\sym^2 f)}{2 \zeta(2)} \right) \frac{1}{u^2}.\]  From the definition of $c_f(\alpha)$ as a sum one sees that $|c'_f(\alpha)| \ll_f \alpha^{-1}$ as $\alpha \rightarrow 0$, so that $\left| \phi' \left(\frac{1}{4\pi X}\right)\right| \ll_f X^3$.  The error term can be treated similarly.   Hence, the difference between the right hand side of the Main Theorem and the right hand side of Corollary 1 is $\ll_f X^{-2}$ with this choice of $\psi$.
\end{proof}
Let \[\Gamma(s,x) = \int_x^\infty e^{-t}t^s\,\frac{dt}{t}\] denote the incomplete gamma function.
\begin{cor} Let
\[\Sigma_f(X,Y) := \sum_{h \leq Y} \sum_{n \geq 1} \lambda_f(n) \lambda_f(n+h) \left(1-\frac{h}{n+h}\right)^{\frac{k-1}{2}} \frac{\Gamma(k-1,(k-1)\frac{n+h}{X})}{\Gamma(k-1)} \]
Then \[ \Sigma_f(X,Y) = -\frac{L(1,\sym^2 f)}{2 \zeta(2)}X + \frac{Y^2}{\Gamma(k-1)} \int_{(k-1)\frac{Y^2}{ X}}^\infty \frac{c_f(u)}{u^2}\,du + O_k\left(X^\frac{1}{2} Y^{\frac{1}{3}(1+\theta)}\right),\] where the integral appearing here either cancels the main term if $\frac{Y^2}{X} $ approaches $0$, or decays rapidly as a function of $\frac{Y^2}{X}$ if this parameter grows without bound. 
\end{cor}
\begin{proof}
Let $\mathds{1}_{>x}$ denote the indicator function of the open set $(x,\infty) \subset \R$.  In similar fashion to the previous proof, let $\psi(y)$ be a smooth approximation to \[\frac{1}{4\pi \Gamma(k-1)} \mathds{1}_{>\frac{k-1}{4 \pi X}}(y),\] and compute the answer.    
\end{proof}

The spectral theorem as used in this paper holds for $L^2$ functions which are also $C^\infty$ and bounded, however, it should also hold for a much wider class of functions, and thus our Main Theorem should in fact hold for a much wider class of functions.  In practice however, one may always obtain a corollary for a specific choice of $\psi$ by elementary arguments similar to the proof of Corollary 1, so we do not pursue the problem of expanding the class of functions for which the Main Theorem holds.  

\bibliographystyle{amsplain}	
\bibliography{myrefs}

\end{document}